\def\draft{n}
\documentclass[12pt]{amsart}
\usepackage[headings]{fullpage}
\usepackage{amssymb,epic,eepic,epsfig}
\usepackage{epstopdf}
\usepackage{graphicx}
\usepackage{texdraw}
\usepackage{url}
\usepackage[bookmarks=true,%
    colorlinks=true,%
    linkcolor=blue,%
    citecolor=blue,%
    filecolor=blue,%
    menucolor=blue,%
    urlcolor=blue,%
    breaklinks=true]{hyperref}

\newtheorem{theorem}{Theorem}[section]
\theoremstyle{definition}

\newtheorem{lemma}[theorem]{Lemma}
\newtheorem{definition}[theorem]{Definition}
\newtheorem{remark}[theorem]{Remark}
\newtheorem{corollary}[theorem]{Corollary}

\newtheorem{question}[theorem]{Question}

\def\printname#1{
        \if\draft y
                \smash{\makebox[0pt]{\hspace{-0.5in}
                        \raisebox{8pt}{\tt\tiny #1}}}
        \fi
}

\newcommand{\psdraw}[2]
         {\begin{array}{c} \hspace{-1.3mm}
        \raisebox{-4pt}{\epsfig{figure=draws/#1.eps,width=#2}}
        \hspace{-1.9mm}\end{array}}

\newlength{\standardunitlength}
\catcode`\@=11
\long\def\@makecaption#1#2{%
     \vskip 10pt

\setbox\@tempboxa\hbox{
       \small\sf{\bfcaptionfont #1. }\ignorespaces #2}%
     \ifdim \wd\@tempboxa >\captionwidth {%
         \rightskip=\@captionmargin\leftskip=\@captionmargin
         \unhbox\@tempboxa\par}%
       \else
         \hbox to\hsize{\hfil\box\@tempboxa\hfil}%
     \fi}
\font\bfcaptionfont=cmssbx10 scaled \magstephalf
\newdimen\@captionmargin\@captionmargin=2\parindent
\newdimen\captionwidth\captionwidth=\hsize
\catcode`\@=12

\def\lbl#1{\label{#1}\printname{#1}}

\def\BN{\mathbb N}
\def\BZ{\mathbb Z}

\def\la{\langle}
\def\ra{\rangle}

\def\longto{\longrightarrow}

\def\ve{\varepsilon}
\def\ccdot{\!\cdot\!\!}

\begin{document}


\title[Alternating knots, planar graphs and $q$-series]{
Alternating knots, planar graphs and $q$-series}
\author{Stavros Garoufalidis}
\address{School of Mathematics \\
         Georgia Institute of Technology \\
         Atlanta, GA 30332-0160, USA \newline 
         {\tt \url{http://www.math.gatech.edu/~stavros }}}
\email{stavros@math.gatech.edu}
\author{Thao Vuong}
\address{School of Mathematics \\
         Georgia Institute of Technology \\
         Atlanta, GA 30332-0160, USA \newline 
         {\tt \url{http://www.math.gatech.edu/~tvuong }}}
\email{tvuong@math.gatech.edu}
\thanks{S.G. was supported in part by a National Science Foundation
grant DMS-0805078. \\
\newline
1991 {\em Mathematics Classification.} Primary 57N10. Secondary 57M25.
}

\date{December 13, 2013}


\begin{abstract}
Recent advances in Quantum Topology assign $q$-series to knots in
at least three different ways. The $q$-series are given by generalized
Nahm sums (i.e., special $q$-hypergeometric sums) and have unknown
modular and asymptotic properties. We give an efficient method to compute 
those $q$-series that come from planar graphs (i.e., reduced Tait graphs
of alternating links) and compute several terms of those series for all 
graphs with at most 8 edges drawing several conclusions.  
In addition, we give a graph-theory proof of a theorem of Dasbach-Lin
which identifies the coefficient of $q^k$
in those series for $k=0,1,2$ in terms of polynomials on the number of 
vertices, edges and triangles of the graph.
\end{abstract}

\maketitle

\tableofcontents

\section{Introduction}
\lbl{sec.intro}

\subsection{$q$-series in Quantum Knot Theory}
\lbl{sub.QT}

Recent developments in Quantum Topology associate $q$-series 
to a knot $K$ in at least three different ways:
\begin{itemize}
\item
via stability of the coefficients of the colored Jones polynomial of $K$,
\item
via the 3D index of $K$,
\item
via the conversion of state-integrals of the quantum dilogarithm to
$q$-series. 
\end{itemize}
The first method is developed of alternating knots in detail, see 
\cite{Ar1,Ar2,AD} and also \cite{GL1}.
The second method uses the 3D index of an ideal triangulation introduced in 
\cite{DGG1,DGG2}, with necessary and sufficient conditions for its
convergence established in \cite{Ga:index} and its topological invariance
(i.e., independence of the ideal triangulation) for hyperbolic 3-manifolds 
with torus boundary proven in \cite{GHRS}. 
The third method was developed in \cite{GK}.

In all three methods, the $q$-series are multi-dimensional $q$-hypergeometric
series of generalized Nahm type; see \cite[Sec.1.1]{GL1}. 
Their modular and the asymptotic properties remains unknown.
Some empirical results and relations among these $q$-series 
are given in \cite{GZ1,GZ2}.

The paper focuses on the $q$-series obtained by the first method.
For some alternating knots, the $q$-series obtained by the first method can
be identified with a finite product of unary theta or false theta series; see
\cite{AD,Andrews}. This was observed independently by the first author
and Zagier in 2011 for all alternating knots in the Rolfsen table \cite{Rf}
up to the knot $8_4$. Ideally, one might expect this to be the case for all
alternating knots. For the knot $8_5$ however, the first 100 terms of its
$q$-series failed to identify it with a reasonable finite
product of unary theta or false
theta series. This computation was performed by the first author at the 
request of Zagier and the result was announced in \cite[Sec.6.4]{Ga:arbeit}.

The purpose of the paper is to give the details of the above computation
and to extend it systematically to all alternating knots and links with
at most 8 crossings. Our computational approach is similar to the 
computation of the index of a knot given in \cite[Sec.7]{GHRS}.

\subsection{Rooted plane graphs and their $q$-series}
\lbl{sub.qplanar} 

By {\em planar graph} we mean an abstract graph, possibly with loops and 
multiple edges, which can be embedded on the plane. A {\em plane graph} 
(also known as a planar map) is an embedding of a planar graph to the plane. 
A {\em rooted plane map} is a plane map together with the choice of a vertex
of the unbounded region. 

In \cite{GL1} Le and the first author introduced a function 
$$
\Phi: \{ \text{Rooted plane graphs}\} \longto \BZ[[q]],
\qquad G \mapsto \Phi_G(q) 
$$ 
For the precise relation between $\Phi_G(q)$ and
the colored Jones function of the corresponding alternating link $L_G$,
see Section \ref{sec.alt}. To define $\Phi_G(q)$, we need to introduce some 
notation. An \emph{admissible state} 
$(a,b)$ of $G$ is an integer assignment $a_p$ for each face $p$ of $G$ and 
$b_v$ for each vertex $v$ of $G$ such that $a_p+b_v\geq 0$ for all pairs $(v,p)$
such that $v$ is a vertex of $p$. For the unbounded face 
$p_\infty$ we set $a_\infty=0$ and thus $b_v=a_\infty+b_v\geq 0$ for all 
$v\in p_\infty$. We also set $b_v=0$ for a fixed vertex $v$ of $p_\infty$. 
In the formulas below, $v,w$ will denote vertices of $G$, 
$p$ a face of $G$ and $p_\infty$ is the unbounded face. 
We also write $v\in p$, $vw\in p$ if $v$ is a vertex and $vw$ is an 
edge of $p$.

For a polygon $p$ with $l(p)$ edges and vertices $b_1,\dots,b_{l(p)}$
in counterclockwise order
$$
\psdraw{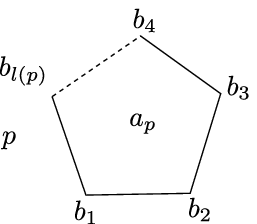}{1in}
$$ 
we define 
\begin{align*}
\gamma(p) &=l(p)a_p^2+2a_p(b_1+b_2+ \dots +b_{l(p)}) \,.
\end{align*}
Let 
\begin{equation}
\lbl{def.A}
A(a,b)= \sum\limits_{p}\gamma(p)+2\sum\limits_{e=(v_iv_j)}b_{v_i}b_{v_j}
\end{equation}
where the $p$-summation (here and throughout the paper) 
is over the set of {\em bounded} 
faces of $G$ and the $e$-summation is over the set of edges 
$e=(v_iv_j)$ of $p$, and
\begin{equation}
\lbl{def.B}
B(a,b)=2\sum\limits_{v}b_v+\sum\limits_{p}(l(p)-2) a_p
\end{equation}
where the $v$-summation is over the set of vertices of $G$ and 
the $p$-summation is over the set of bounded faces of $G$.

\begin{definition}
\lbl{def.phi}\cite{GL1}
With the above notation, we define
\begin{equation}
\lbl{eq.defphi}
\Phi_G(q)=(q)_\infty^{c_2}\sum\limits_{(a,b)}(-1)^{B(a,b)} 
\frac{q^{\frac{1}{2}A(a,b)+\frac{1}{2}B(a,b)}}{\prod\limits_{(p,v): v\in p}(q)_{a_p+b_v}}
\end{equation}
where the sum is over the set
of all admissible states $(a,b)$ of $G$, and in the product 
$(p,v): v\in p$ means a pair of face $p$ and vertex $v$ such that 
$p$ contains $v$. Here, $c_2$ is the number of edges of $G$ and
$$
(q)_\infty=\prod_{n=1}^\infty (1-q)^n
=1 - q - q^2 + q^5 + q^7 - q^{12} - q^{15} \dots
$$
\end{definition}

Convergence of the $q$-series of Equation \eqref{eq.defphi} in the 
formal power series ring $\BZ[[q]]$ is not
obvious, but was shown in \cite{GL1}. Below, we give effective (and
actually optimal) bounds for convergence of $\Phi_G(q)$. To phrase them,
let $b_p=\min\{b_v:v \in p\}$ where $p$ denotes a face of $G$.

\begin{theorem}
\lbl{thm.2}
\rm{(a)} We have 
{\small
\begin{align}
\lbl{eq.Q2}
A(a,b)& =  \sum\limits_{p} \left(
l(p) (a_p+b_p)^2+2 (a_p+b_p)
\sum\limits_{v \in p}(b_v-b_p)
\right.\\ \notag & \left.
+ \sum\limits_{vv'\in p}(b_v-b_p)(b_{v'}-b_p) \right)
+  \sum\limits_{vv'\in p_\infty}b_vb_{v'}  \,.
\end{align}
}
Each term in the above sum is manifestly nonnegative.
\newline
\rm{(b)} $B(a,b)$ can also be written as a finite sum of manifestly
nonnegative linear forms on $(a,b)$.
\newline
\rm{(c)} 
If $\frac{1}{2}(A(a,b)+B(a,b))\leq N$ for some natural number $N$, then for 
every $i$ and every $j$ there exist $c_i, c'_i$ and $c_j,c'_j$ 
(computed effectively from $G$) such that
$$
c_i N  \leq b_i  \leq c_i' N, \qquad\qquad
c_j' \sqrt{N}  \leq a_j  \leq c_j N  + c_j' \sqrt{N} \,.
$$
\end{theorem}
For a detailed illustration of the above Theorem, see Section \ref{sub.L8a7}.

\subsection{Properties of the $q$-series of a planar graph}
\lbl{sub.properties}

The next lemma summarizes some properties of the series $\Phi_G(q)$.
Part (a) of the next lemma is taken from 
\cite[Thm.1.7]{GL1} \cite[Lem.13.2]{GL1}. Parts (b) and (c) were observed
in \cite{AD} and \cite{GL1} and follow easily from the behavior of the
colored Jones polynomial under disjoint union and under a connected sum.
Note that we use the normalization that the colored Jones polynomial of the
unknot is $1$. Part (d) was proven in  \cite{AD} and \cite[Lem.13.3]{GL1}.

\begin{lemma} 
\cite{AD,GL1}
\lbl{lem.reduced}
\rm{(a)} The series $\Phi_G(q)$ depends only on the abstract planar graph $G$
and not on the rooted plane map.
\newline
\rm{(b)} If $G=G_1 \sqcup G_2$ is disconnected, then 
$$
(1-q) \Phi_{G}(q)= \Phi_{G_1}(q) \Phi_{G_2}(q) \,.
$$
\newline
\rm{(c)} If $G$ has a separating edge (also known as a bridge) $e$
and $G\setminus\{e\}=G_1 \sqcup G_2$, then
$$
\Phi_G(q)= \Phi_{G_1}(q) \Phi_{G_2}(q) \,.
$$
\newline
\rm{(d)}
If $G$ is a planar graph (possibly with multiple edges and loops) and
$G'$ denotes the corresponding simple graph obtained by removing all loops
and replacing all edges of multiplicity more than with edges of multiplicity
one, then 
$$
\Phi_G(q)=\Phi_{G'}(q) \,.
$$
\end{lemma}

So, we can focus our attention to simple, connected planar graphs. 
In the remaining of the paper, unless otherwise stated, $G$ will denote a
{\em simple} planar graph.
Let $\langle{f(g)}\rangle_k$ denote the coefficient of $q^k$ of $f(q) 
\in \BZ[[q]]$. The next theorem was proven in \cite{DL} using properties
of the Kauffman bracket skein module. We give an independent proof using
combinatorics of planar graphs in Section \ref{sec.thm1}. Our proof
allows us to compute the coefficient of $q^3$ in $\Phi_G(q)$, observing
a new phenomenon related to induced embeddings, and guess the coefficients
of $q^4$ and $q^5$ in $\Phi_G(q)$. This is discussed in a subsequent
publication \cite{GNZ}.

\begin{theorem}
\lbl{thm.1}\cite{DL}
If $G$ is a planar graph, we have
\begin{subequations}
\begin{align}
\lbl{eq.coeff0}
\langle{\Phi_G(q)}\rangle_0 &= 1 \\
\lbl{eq.coeff1}
\langle{\Phi_G(q)}\rangle_1 &= c_1 -c_2 -1 \\
\lbl{eq.coeff2}
\langle{\Phi_G(q)}\rangle_2&= \frac{1}{2}\left( (c_1 -c_2)^2 -2 c_3 - c_1 +
c_2 \right) 
\end{align}
\end{subequations}
where $c_1$, $c_2$ and $c_3$ denotes the number of vertices, edges and 
3-cycles of $G$.
\end{theorem}

If $G_1$ and $G_2$ are two planar graphs with distinguished boundary
edges $e_1$ and $e_2$, let $G_1 \cdot G_2$ denote their edge connected
sum along $e_1=e_2$ depicted as follows:
$$
\psdraw{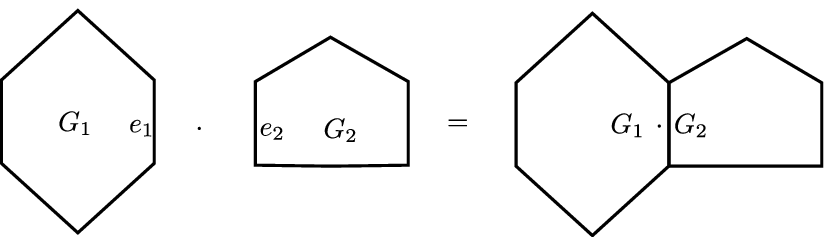}{3.3in}
$$
Let $P_r$ denote a planar polygon with $r$ edges when $r \geq 3$
and let $P_2$ denote the connected graph with two vertices and one edge,
a reduced form of a bigon. For a positive natural 
number $b$, consider the unary theta (when $b$ is odd) and false theta 
series (when $b$ is even) $h_b(q)$ given by
\begin{equation*}
\lbl{eq.ha}
h_b(q)=\sum_{n \in \BZ} \ve_b(n) \, q^{\frac{b}{2} n(n+1)-n}
\end{equation*}
where
\begin{equation*}
\ve_b(n) =\begin{cases}
(-1)^n  & \text{if $b$ is odd} \\
1       & \text{if $b$ is even and $n \geq 0$} \\
-1       & \text{if $b$ is even and $n<0$} 
\end{cases}
\end{equation*}
Observe that 
$$
h_1(q)=0, \qquad h_2(q)=1, \qquad h_3(q)=(q)_\infty \,.
$$
The following lemma (observed independently by Armond-Dasbach) 
follows from the Nahm sum for $\Phi_G(q)$ combined with a $q$-series 
identity (see Equation \eqref{eq.Sbb} below). This identity was proven by 
Armond-Dasbach \cite[Thm.3.7]{AD} 
and Andrews \cite{Andrews}.
 
\begin{lemma}
\lbl{lem.sum}
For all planar graphs $G$ and natural numbers $r \geq 3$ we have:
$$
\Phi_{G \cdot P_r}(q) = \Phi_G(q) \Phi_{P_r}(q) = \Phi_G(q) h_r(q) \,.
$$
\end{lemma}

\begin{question}
\lbl{que.1}
Is it true that for all planar graphs $G_1$ and $G_2$ we have:
$$
\Phi_{G_1 \cdot G_2}(q) = \Phi_{G_1}(q) \Phi_{G_2}(q) ?
$$
\end{question}

As an illustration of Lemma \ref{lem.sum}, 
for the three graphs of Figure \ref{f.triple.knots}, we have:
$$
\Phi_{L8a8}(q)=\Phi_{8_{13}}(q)=h_4(q)h_3(q)^2 \,.
$$

\begin{figure}[htpb]
$$
\psdraw{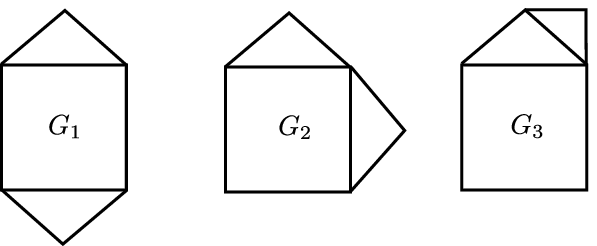}{3in}
$$
\caption{Three graphs $G_1$, $G_2$, $G_3$ and the corresponding alternating
links $L8a8$, $L8a8$ and $8_{13}$.}
\lbl{f.triple.knots}
\end{figure}

\begin{remark}
\lbl{exer.1}
Observe that the alternating planar projections of the graphs $G_1$
and $G_2$ of Figure \ref{f.triple.knots} are related by a {\em flype move}
\cite[Fig.1]{MT}.

Flyping a planar alternating link projection corresponds to the operation
on graphs shown in Figure \ref{f.flyping.graphs}.

If the planar graphs $G$ and $G'$ are related by flyping, then
$\Phi_G(q)=\Phi_{G'}(q)$, since the corresponding alternating links are
isotopic. 
\end{remark}

\begin{figure}[htpb]
$$
\psdraw{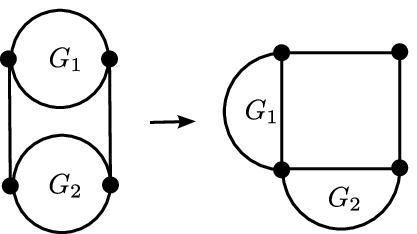}{2in}
$$
\caption{A flyping move on a planar graph.}
\lbl{f.flyping.graphs}
\end{figure}

\begin{remark}
\lbl{rem.G9.10.12.16}
Theorem \ref{thm.1} might temp one to conjecture that $\Phi_G(q)$ depends
on the number of vertices and edges of $G$ and on the number of $k$-faces
of $G$ for $k \geq 3$. This is not true. For example, consider the three
graphs $G^9_{9}$, $G^9_{11}$ and $G^9_{13}$ of Figure \ref{f.graphs9}
shown here: 
$$
\psdraw{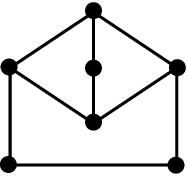}{0.7in} \qquad
\psdraw{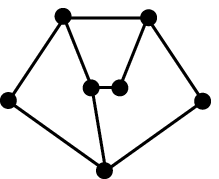}{0.8in} \qquad
\psdraw{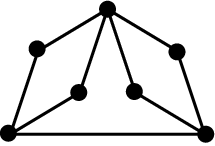}{0.7in}
$$
All three graphs have $7$ vertices, $9$ edges, $2$ square
faces and $2$ pentagonal faces.
The DT codes of the corresponding links are given by:

\begin{center}
\begin{tabular}{|c|l|}
\hline
$G^9_{10}$ & DTCode[$\{$16, 10, 14, 12, 2, 18, 6$\}$, $\{$4, 8$\}$] \\ \hline
$G^9_{12}$ & DTCode[$\{$6, 10, 14, 18, 4, 16, 8, 2, 12$\}$] \\ \hline
$G^9_{16}$ & DTCode[$\{$6, 10$\}$, $\{$4, 12, 18, 2, 16$\}$, $\{$8, 14$\}$] 
\\ \hline
\end{tabular}
\end{center}

\end{remark}
On the other hand, the colored Jones function
of the corresponding alternating links \cite{B-N} gives that 
$$
\begin{array}{|c|l|} \hline
G & \Phi_G(q) \\  \hline
G^9_{10} & 1 - 3 q + 3 q^2 + 2 q^3 - 7 q^4 + 3 q^5 + \dots \\ \hline
G^9_{12} & 1 - 3 q + 3 q^2 + q^3 - 7 q^4 + 6 q^5 + \dots \\ \hline
G^9_{16} & 1 - 3 q + 3 q^2 + q^3 - 8 q^4 + 6 q^5 + \dots \\ \hline
\end{array}
$$


\section{The connection between $\Phi_G(q)$ and alternating 
links}
\lbl{sec.alt}

In this section explain connection between $\Phi_G(q)$ and the
colored Jones function of the alternating link $L_G$
following \cite{GL1}.

\subsection{From planar graphs to alternating links}

Given a planar graph $G$ (possibly with loops or multiple edges), 
there is an alternating planar projection 
of a link $L_G$ given by:
$$
\psdraw{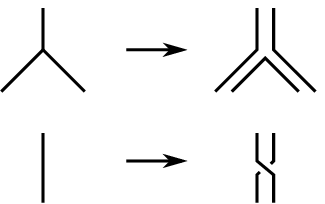}{2in}
$$

\subsection{From alternating links to planar (Tait) graphs}
\lbl{subsection.tait}

Given a diagram $D$ of a {\it reduced alternating non-split} link $L$, its 
Tait graph can be constructed as follows: the diagram $D$ gives rise to a 
polygonal complex of $\mathbb{S}^2=\mathbb{R}^2\cup \{\infty\}$. Since $D$ is 
alternating, it is possible to label each polygon by a color $b$ (black) or $w$
(white) such that at every crossing the coloring looks as follows:

\begin{figure}[htb]
\includegraphics[scale=1]{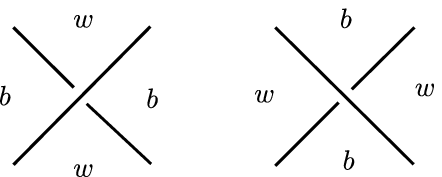}
\caption{The checkerboard coloring of a link diagram}
\end{figure}

There are exactly two ways to color the regions of $D$ with black and white
colors. In this note we will work with the one whose unbounded region has 
color $w$. In each $b$-colored polygon (in short, $b$-polygon) 
we put a vertex and connect two of them 
with an edge if there is a crossing between the corresponding polygons. 
The resulting graph is a planar graph called the Tait graph associated with 
the link diagram $D$. Note that the Tait graph is always planar but not
necessarily reduced. Although the reduction of the Tait graph may change
the alternating link and its colored Jones polynomial, it does not
change the limit of the shifted colored Jones function in Theorem 
\ref{thm.phi0} because of Lemma \ref{lem.reduced}.

\subsection{The limit of the shifted colored Jones function}
\lbl{sub.0limit}

When $L$ is an alternating link, 
the colored Jones polynomial $J_{L,n}(q) \in \BZ[q^{\pm \frac{1}{2}}]$ 
(normalized to be $1$ at the unknot, and colored by the $n$-dimensional 
irreducible representation of $\mathfrak{sl}_2$ \cite{GL1}) 
has lowest $q$-monomial with coefficient $\pm 1$, and after dividing by 
this monomial, we obtain the {\em shifted} colored Jones polynomial 
$\hat J_{L_G,n}(q) \in 1 + q \BZ[q]$. Let 
$\la f(q) \ra_N$ denotes the coefficient of $q^N$ in $f(q)$.
The limit $f(q)=\lim_n f_n(q) \in \BZ[[q]]$ of a sequence of polynomials
$f_n(q) \in \BZ[q]$ is defined as follows \cite{GL1}. For every natural
number $N$, there exists a natural number $n_0(N)$ such that
$\la f_n(q) \ra_N=\la f(q) \ra_N$ for all $n \geq n_0(N)$.

\begin{theorem}
\lbl{thm.phi0}\cite[Thm.1.10]{GL1}
Let $L$ be an alternating link projection and $G$ be its Tait graph. 
Then the following limit exists
\begin{equation}
\lbl{eq.phiJ}
\lim_{n \to \infty} \hat J_{L,n}(q)=\Phi_G(q) \in \BZ[[q]]
\end{equation}
\end{theorem}

\begin{remark}
\lbl{rem.reduced}
\rm{(a)}
The convergence statement in the above theorem holds in the following 
strong form \cite{GL1}: for every natural number $N$, and for $n > N$
we have:
\begin{equation}
\lbl{eq.Phiconv}
\la \hat J_{L,n}(q) \ra_N = \la \Phi_G(q) \ra_N \,.
\end{equation}
\rm{(b)}
$\Phi_G(q)$ is the \emph{reduced} version of the one in 
\cite[Thm.1.10]{GL1} and differs from the unreduced version
$\Phi^{\text{TQFT}}_G(q)$ by 
$$
\Phi_G(q)=(1-q)\Phi^{\text{TQFT}}_G(q) \,,
$$
where
\begin{equation}
\lbl{phi.TQFT}
\Phi^{\text{TQFT}}_G(q) = (q)_\infty^{c_2}\sum\limits_{(a,b)}(-1)^{B(a,b)} 
\frac{q^{\frac{1}{2}A(a,b)+\frac{1}{2}B(a,b)}}{\prod\limits_{(p,v): v\in p}(q)_{a_p+b_v}}
\end{equation}
and the summation $(a,b)$ is over all admissible states where we
do not assume that $b_v=0$ for a fixed vertex $v$ in the unbounded face of
$G$.
\end{remark}


\section{Proof of Theorem \ref{thm.2}}
\lbl{sec.thm2}

In this section we prove Theorem \ref{thm.2}. Part (a) follows from
completing the square in Equation \eqref{def.A}:

\begin{align*}
A(a,b) &= \sum\limits_p (l(p)a_p^2+2a_p(\sum\limits_{v\in p}b_v))
+2\sum\limits_{e=(v_iv_j)}b_{v_i}b_{v_j}\\
&=  \sum\limits_p (l(p)(a_p+b_p)^2+2a_p(
\sum\limits_{v\in p}b_v-l(p)b_p)-l(p)b_p^2+2\sum\limits_{e=(v_iv_j)}b_{v_i}b_{v_j})\\
&= \sum\limits_p (l(p)(a_p+b_p)^2+2(a_p+b_p)(\sum\limits_{v\in p}b_v-l(p)b_p)
+\sum\limits_{e=(v_iv_j)\in p}(b_{v_i}-b_p)(b_{v_j}-b_p))\\
& +  \sum\limits_{e=(v_iv_j)\in p_\infty}b_{v_i}b_{v_j}
\end{align*}

For the remaining parts of Theorem \ref{thm.2}, fix a $2$-connected 
planar graph $G$, a vertex $v_0$ of $G$ and a bounded face $p_0$ 
of $G$ that contains $v_0$. 

\begin{lemma}
\lbl{lem.dist}
There exists a graph $\Gamma$ which depends on $G,v_0,p_0$ 
such that:
\begin{itemize}
\item The vertices of $\Gamma$ are vertices of $G$ as well as one 
vertex $v_p$ for each bounded face $p$ of $G$.
\item The edges of $\Gamma$ are of the form $vv_p$ where $v$ is a 
vertex of $G$ and $p$ is a bounded face that contains $v$. 
\item $v_0v_{p_0}$ is an edge of  $\Gamma$.
\item Every vertex $v$ in $G$ has degree $n_v$ in $\Gamma$ where
\begin{equation*}
n_v= \begin{cases} 
2 ~ & \text{if $v$ is not a boundary vertex}\\
\leq 2 ~ &  \text{if $v$ is a boundary vertex}
\end{cases}
\end{equation*}
\end{itemize}
\end{lemma}
\begin{proof}
First note we can assume that each face $p$ of $G$ is a triangle. Indeed, 
if a face $p$ is not a triangle, we can divide it into a union of triangles 
by creating new edges inside $p$. Once we have succeeded in constructing a 
$\Gamma$ for the resulted graph, we can remove the added edges in $p$ and 
collapse all the interior vertices of the newly created triangles in $p$ 
into one single vertex $v_p$. The figures below illustrate the above process.
$$
\psdraw{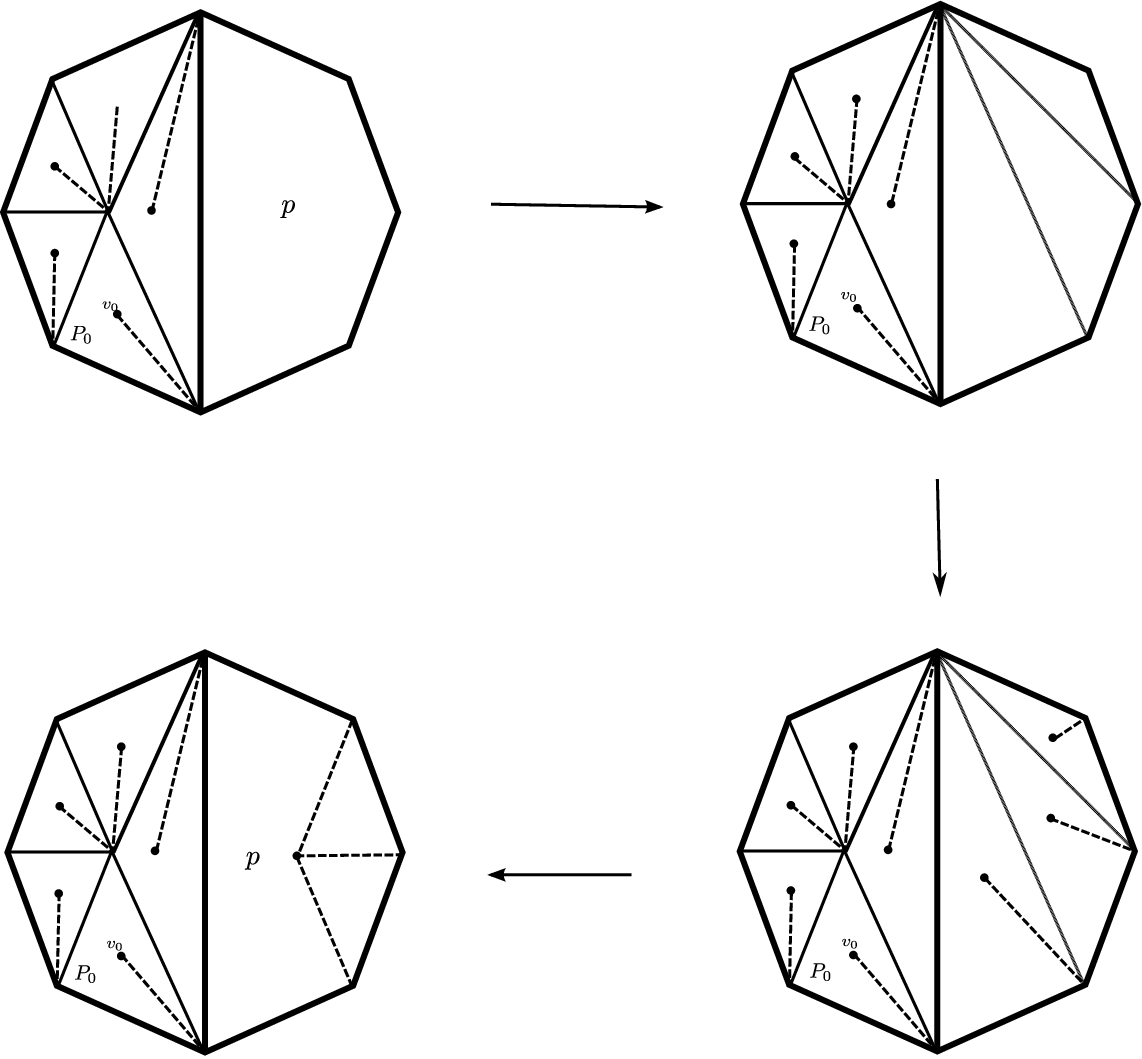}{3in}
$$
Now assuming that all faces of $G$ are triangles, let us proceed by 
induction on the number of vertices of $G$. If there is no interior vertex 
in $G$ then  since the unbounded face $p_\infty$ is also a triangle, $G$ 
itself is a triangle and we are done. Therefore let us assume that there is 
an interior vertex $v$ of $G$. Locally the graph at $v$ looks like the 
following:
$$
\psdraw{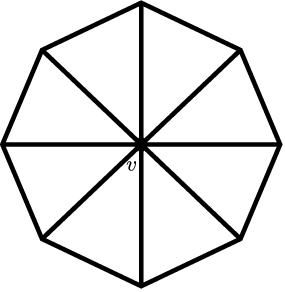}{1in}
$$
Next we remove $v$ and all of the edges incident to it from $G$ and denote 
the resulted face by $p$. Let $w$ be a vertex of $p$ and connect $w$ to 
each of the vertices of $p$ by an edge. Denote the resulted graph by $G_{w}$. 
By induction hypothesis, there exists a graph $\Gamma_w$
for $G_{w}$. At $w$ make another copy of the vertex called $w'$. Now drag 
$w'$ into the interior of $p$ while keeping it connected to vertices of $p$ 
and at the same time delete the edges that are incident to $w$ and that lie 
in the interior of $p$. This has to be done in such a way that all the 
vertices of $\Gamma_w$ still lie in the interior of the new triangles that 
have $w'$ as a vertex. Create two new vertices in the interior of the two
triangles in $p$ that contain $w$ as a vertex and connect them to 
$w'$. The resulted graph satisfies the requirements of the lemma. 
The figures below explain the process.
$$
\psdraw{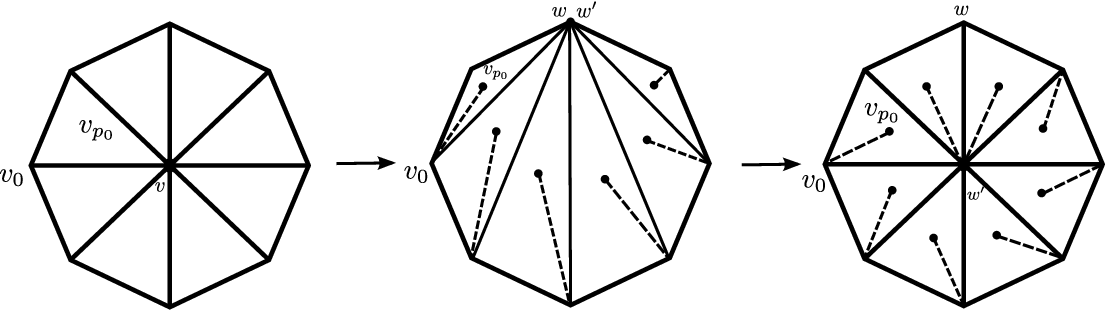}{4in}
$$
\end{proof}

\begin{proof}(of part (b) of Theorem \ref{thm.2})
We can decompose $B(a,b)$ into a finite sum of nonnegative terms as follows
\begin{equation}
\lbl{eq.dist}
B(a,b)=\sum\limits_{\hat{e}=(vv_p)}(a_p+b_v)+\sum\limits_v (2-n_v)b_v
\end{equation}
where the summation is over all edges of $\Gamma$.
\end{proof}

\begin{corollary}
\lbl{cor.2}
For a pair $(p,v)$ where $p$ is a face of $G$ and $v$ is a vertex  of $p$ 
then $B(a,b)\geq a_{p}+b_{v}$.
\end{corollary}

\begin{proof}
This is a direct consequence of Equation \eqref{eq.dist} since by Lemma 
\ref{lem.dist} there exists 
a graph $\Gamma$ that contains $vv_{p}$ as an edge.
\end{proof}

\begin{proof}(of part (c) of Theorem \ref{thm.2})
Let us prove the linear bound on the $b_v$ first. Let us
 set $b_{v_0}=0$ where $v_0$ is a boundary vertex of $G$. Let 
$p_0$ be a bounded face that contains $v_0$, so we have $a_{p_0}+b_{v_0}\geq 0$ .
Since $0\leq B(a,b)\leq 2N$ by part (b) of Theorem \ref{thm.2} and Corollary 
\ref{cor.2} we have that $0\leq a_{p_0}+b_{v_0}\leq 2N$. Since $b_{v_0}=0$ this 
means that 
$0\leq a_{p_0}\leq 2N$. Similarly if $v$ is another vertex of $p_0$ then by 
Corollary
\ref{cor.2} we have  $0\leq a_{p_0}+b_{v}\leq 2N$ which 
implies that $-2N\leq b_v\leq 2N$. Let $G'$ be the graph obtained from $G$ 
by removing the boundary edges of $p_0$. Choose a face $p'$ of $G'$ and a 
vertex $v'\in p'$ that also belongs to the removed face $p_0$. Repeat the 
above process with $(p',v')$ we have that $-4N\leq b_{v''}\leq 4N$ for any 
$v''\in p'$. Continuing this process until all faces of $g$ are covered
have that $|b_v|\leq dN$ for all vertices $v$ of $G$.

To prove the bound for the $a_p$'s, note that from part (a) of Theorem
\ref{thm.2} we have that $\frac{e(p)}{2}(a_p+b_v)^2\leq N$ for all 
bounded faces $p$ and all vertices $v$ of $G$. This implies that 
$|a_p+b_v|\leq \sqrt{\frac{2}{e_p}}\sqrt{N}$. Since $|b_v|\leq dN$ this 
implies that $|a_p|\leq \sqrt{\frac{2}{e_p}}\sqrt{N}+dN$. For the lower bound 
of $a_p$, note that since $a_p+b_v\geq 0$ we have $a_p\geq -b_v\geq -dN$.
\end{proof}


\section{The coefficients of $1$, $q$ and $q^2$ in $\Phi_G(q)$}
\lbl{sec.thm1}

\subsection{Some lemmas}
\lbl{sub.lemmas}

In this section we prove Theorem \ref{thm.1}, using the unreduced series
$\Phi^{\text{TQFT}}_G(q)$ of Equation \eqref{phi.TQFT}. 
Our admissible states $(a,b)$ in this section do not satisfy the property
that $b_v=0$ for some vertex $v$ of the unbounded face of $G$.

Since $A(a,b)+B(a,b) \geq 0$ for an admissible state $(a,b)$ with equality
if and only if $(a,b)=(0,0)$ (as shown in Theorem \ref{thm.2}), 
it follows that the coefficient of $q^0$ in 
$\Phi_G(q)$ is $1$. For the remaining of the proof of Theorem \ref{thm.1}
we will use several lemmas.

\begin{lemma}
\lbl{lem.triangle}
Let $G$ be a $2$-connected planar graph whose unbounded face has $V_\infty$ 
vertices. If $(a,b)$ is an admissible state such that 
\begin{enumerate}
\item $b_v=b_{v'}=1$ where $vv'$ is an edge of $p_\infty$,
\item $a_p+b_p=0$ for any face $p$ of $G$,
\item $(b_{v_1}-b_p)(b_{v_2}-b_p)=0$ for any face $p$ of $G$ and edge $v_1v_2$ 
of $p$,
\end{enumerate}
then 
\begin{itemize}
\item
$b_v \geq 1$ for all vertices $v$, 
\item
$a_p=-1$ for all faces $p\neq p_{\infty}$, and 
\item
$B(a,b)\geq 2+V_\infty$.
\end{itemize}
\end{lemma}

\begin{proof}
Let $p$ be the bounded face that contains $v,v'$. We have 
$(b_v-b_p)(b_{v'}-b_p)=0$ so $b_p=1$ since $b_v=b_{v'}=1$. (2) then implies 
that $a_p=-b_p=-1$ and thus $b_w\geq b_p=1$ for all $w\in p$.
\begin{figure}[htpb]
 $$\psdraw{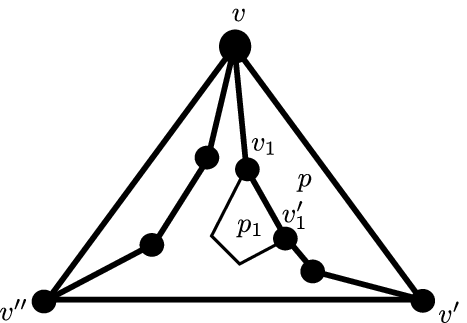}{1.5in}$$
\end{figure}
Let $v_1v_1'$ be another edge of $p$ and let $p_1\neq p$ bea bounded face 
that contains $v_1v_1'$. Since $(b_{v_1}-b_p)(b_{v_1'}-b_p)=0$ we have 
$\min\{b_{v_1},b_{v_1'}\}=b_p=1$. So from $(b_{v_1}-b_{p_1})(b_{v_1'}-b_{p_1})=0$ 
we have that $b_{p_1}=1$. Therefore $a_{p_1}=-1$ and $b_w\geq b_{p_1}=1$ for 
any vertex $w\in p_1$. By a similar argument we can show that $b_v\geq 1$ 
for every vertex $v$ and $a_p=-1$ for every face $p$ of $G$. Let 
$p_1,p_2, \dots, p_f$  be the bounded faces of $G$, where $f=F_G-1$. Then 
from Equation \eqref{def.B} we have
\begin{align*}
B(a,b) & = -\sum\limits_{j=1}^f (l(p_j)-2)+2\sum\limits_{v}b_v\\
& \geq  -\sum\limits_{j=1}^f l(p_j)+2f+2c_1\\
& =  -(2c_2-V_\infty) +2F_G-2+ 2c_1\\
& =  2(c_1-c_2+F_G)-2+V_\infty\\
& =  2+V_\infty
\end{align*}
\end{proof}

The proof of the next lemma is similar to the one of Lemma \ref{lem.triangle}
and is therefore omitted.

\begin{lemma}
\lbl{lem.triangle2}
Let $G$ be a $2$-connected planar graph whose unbounded face has 
$V_\infty$ vertices. If $(a,b)$ is an admissible state such that 
\begin{enumerate}
\item $b_v=b_{v'}=0$ and $(b_v-b_p)(b_{v'}-b_p)=1$ where $p$ is a boundary 
face and $vv'$ is a boundary
 edge that belongs to $p$,
\item $a_p+b_p=0$ for any face $p$ of $G$,
\item $(b_{v_1}-b_p)(b_{v_2}-b_p)=0$ for any face $p$ of $G$ and edge 
$v_1v_2$ not 
on the boundary of $p$.
\end{enumerate}
Then $b_w \geq -1$ for all vertices $w$, $a_p=1$ for all 
faces $p\neq p_{\infty}$ and $B(a,b)\geq V_\infty-2$.  Furthermore $B(a,b) = V_\infty-2$ if and only if 
\begin{itemize}
\item $b_v=0$ for all boundary vertices $v$ and $b_w=-1$ for all other vertices $w$.
\item $a_p=1$ for all faces $p$.
\end{itemize}
\end{lemma}

\begin{lemma}
\lbl{lem.induction}
Let $G$ be a  $2$-connected planar graph, $p_0$ be a boundary face and 
$(a,b)$ be an admissible state such that
\begin{enumerate}
\item
$a_{p_0}+b_{p_0}=0$,
\item
There exists a boundary edge  $vv'$ of $p_0$ such that $b_vb_{v'}=0$ and 
$(b_{v}-b_{p_0})(b_{v'}-b_{p_0})=0$, 
\item
Let $G_0$ be the graph obtained from $G$ by deleting the boundary edges of 
$p_0$ and let $(a_0,b_0)$ be the restriction of the admissible state $(a,b)$ 
on $G_0$.
\end{enumerate} 
Then, 
\begin{itemize}
\item[(a)] 
$(a_0,b_0)$ is an admissible state for $G_0$,
\item[(b)]
$A(a_0,b_0)=A(a,b)-\sum\limits_{e=(vv'):v,v'\in p_0\cap p_\infty}b_vb_{v'}$,
\item[(c)]
$B(a_0,b_0)=B(a,b)-2\sum\limits_{v\in V_0}b_v$, where $V_0$ is the set of 
boundary vertices of $p_0$ that do not belong to any other bounded face,
\item[(d)]
$B(a,b)\geq 2\sum\limits_{v\in V_0}b_v$,
\item[(e)]
If furthermore $B(a,b)\leq 1$ then $A(a,b)=A(a_0,b_0), B(a,b)=B(a_0,b_0)$.
\end{itemize}
\end{lemma}

\begin{proof}
From (2) we have either $b_v=0$ or $b_{v'}=0$ and  
it follows from $(b_{v}-b_{p_0})(b_{v'}-b_{p_0})=0$ that $b_{p_0}=0$. This means
that  we have $b_v\geq 0$ for all $v\in p_0$. This 
implies (a). Furthermore (1) implies that $a_{p_0}=0$ and thus 
$A(a,b)-A(a_0,b_0)=l(p_0)a_{p_0}^2+2a_{p_0}(\sum\limits_{v\in p_0}b_v)
+\sum\limits_{e=(vv'):v,v'\in p_0\cap p_\infty}b_vb_{v'}=
\sum\limits_{e=(vv'):v,v'\in p_0\cap p_\infty}b_vb_{v'}$ and 
$B(a,b)-B(a_0,b_0)=a_{p_0}+2\sum\limits_{v\in V_0}b_v=2\sum\limits_{v\in V_0}b_v$. 
This proves (b) and (c). (d) follows from (c) since we have 
$0\leq B(a_0,b_0)=B(a,b)-2\sum\limits_{v\in V_0}b_v$, (e) is a consequence of (b), (c)
and (d) since $1\geq B(a,b)\geq 2\sum\limits_{v\in V_0}b_v$ implies 
that $\sum\limits_{v\in V_0}b_v=0$.
$$
\psdraw{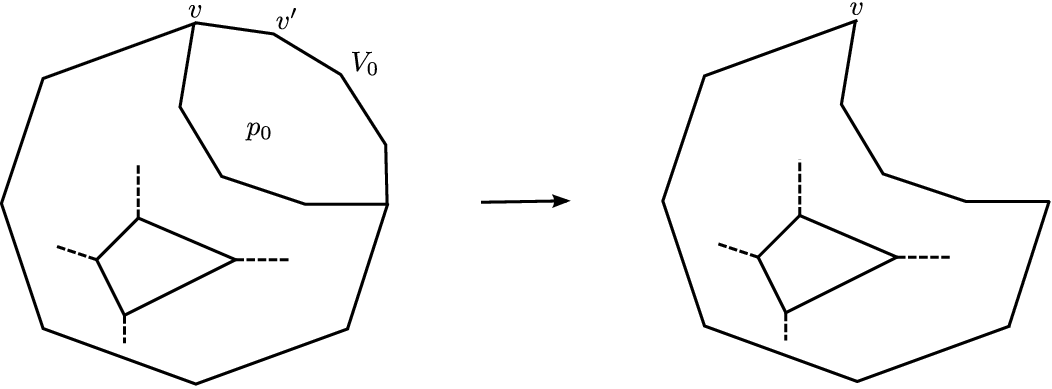}{3in}
$$
\end{proof}

\subsection{The coefficient of $q$ in $\Phi_G(q)$}
\lbl{coef.q}

We need to find the admissible states $(a,b)$ such that 
$\frac{1}{2}(A(a,b)+B(a,b))=1$. 
Parts (a) and (b) of Theorem \ref{thm.2} imply that 
$A(a,b), B(a,b)\in \BN$. Thus, if $\frac{1}{2}(A(a,b)+B(a,b))=1$ then 
we have the following cases:
$$
\begin{array}{|c|c|c|c|}
\hline
A(a,b) & 2 & 1 & 0 \\ \hline
B(a,b) & 0 & 1 & 2 \\ \hline
\end{array}
$$

{\bf Case 1:} $(A(a,b), B(a,b))=(2,0)$. Since $l(p)\geq 3$, we should 
have $a_p+b_p=0$ for all faces $p$. This implies that 
$a_p+b_v=a_p+b_p+b_v-b_p=b_v-b_p$ and it follows from Corollary \ref{cor.2} 
that $0=B(a,b)\geq a_p+b_v = b_v-b_p$. This means $b_v-b_p=a_p+b_v=0$ for all 
faces $p$ and vertices $v$ of $p$, so Equation \eqref{eq.Q2} is equivalent to
\begin{equation}
\lbl{eq.10}
\sum\limits_{vv'\in p_\infty}b_vb_{v'}=2 \,.
\end{equation}
If $vv'$ is an edge of $G$ and $p$ is a face that contains $vv'$ then we 
have $a_p+b_v=0=a_p+b_{v'}$ and therefore $b_v=b_{v'}$. So by Equation 
\eqref{eq.10} there exists a boundary edge $vv'$ such that $b_{v}=b_{v'}=1$. 
Lemma \ref{lem.triangle} implies that $B(a,b)\geq 2+V_\infty>0$ which is 
impossible. Therefore there are no admissible states $(a,b)$ that 
satisfy $(A(a,b), B(a,b))=(2,0)$.\\

{\bf Case 2:} $(A(a,b), B(a,b))=(1,1)$.  As above we have that $a_p+b_p=0$ 
for all faces $p$. Since $A(a,b)=1$, there is either a bounded face $p_1$ 
with an edge $v_1v_1'$ such that $(b_{v_1}-b_{p_1})(b_{v_1'}-b_{p_1})=1$  or a 
boundary edge $v_2v_2'$ such that $b_{v_2}b_{v_2'}=1$  and all other terms in 
Equation \eqref{eq.Q2} are equal to zero. Let $p_2$ be the bounded face 
that contains $v_2v_2'$ and let $p\neq p_1,p_2$ be a bounded face. Let 
$G'$ be the graph obtained from $G$ by deleting the boundary edges of 
$p$ and $(a',b')$ be the restriction of $(a,b)$ on $G'$. By part (e)
of Lemma \ref{lem.induction}, we have $A(a',b')=A(a,b)$ and $B(a',b')=B(a,b)$. 
Continue this process until either $G=p_1$ or $G=p_2$. If $G=p_2$ then  
$b_{v_2}b_{v_2'}=1$ and therefore $B(a,b)\geq 2(b_{v_2}+b_{v_2}')= 4$ which is 
impossible. If $G=p_1$ then $v_1,v_2$ are now boundary vertices and so $b_{v_1}b_{v_1'}=0$ and we can assume that $b_{v_1}=0$. But this implies that $-b_{p_1}(b_{v_1'}-b_{p_1})=1$ hence $b_{p_1}=-1$. This is impossible since $b_{p_1}$ is a boundary vertex. Thus there are no 
admissible states $(a,b)$ that satisfy $(A(a,b), B(a,b))=(1,1)$. 

{\bf Case 3:} $(A(a,b), B(a,b))=(0,2)$. Since $A(a,b)=0$ we should have 
\begin{itemize}
\item $a_p+b_p=0$ for all faces $p$,
\item $b_vb_{v'}=0$ for all boundary edges $vv'$,
\item $(b_v-b_p)(b_{v'}-b_p)=0$ for all bounded faces $p$ and and edges 
$vv'\in p$.
\end{itemize} 
Let $p$ be a bounded face of $G$. Let $G'$ be the graph obtained from $G$ 
by deleting the boundary edges of $G$ and $(a',b')$ be the restriction of 
$(a,b)$ on $G'$. By part (e) of Lemma \ref{lem.induction},
we have $A(a',b')=A(a,b)$ and
$B(a',b')=B(a,b)-2n_p$ where $n_p\in \mathbb{N}$. Since $B(a,b)=2$, 
$n_p\leq 1$ and $n_p=1$ if and only if there exists exactly one boundary 
vertex $v\in p$ such that $b_v=1$ and $b_v'=0$ for any other boundary vertex 
$v'$ of $p$.
Continuing this process it is easy to show that an admissible state 
$(a,b)$ such that $(A(a,b), B(a,b))=(0,2)$ must satisfy the following: 
\begin{itemize}
\item $a_p=0$ for all $p$,
\item $b_v=1$ for a vertex $v$ and $b_{v'}=0$ for any other  vertex $v'$ of $G$. 
\end{itemize}
 The contribution of this state to $\Phi_G(q)$ is 
$\frac{q}{(1-q)^{\text{deg}(v)}}=q +O(q^2)$.

Thus from Theorem \ref{thm.phi0} and cases 1-3 we have
\begin{align*}
\langle \Phi^{\text{TQFT}}_G(q) \rangle_1 &=
\Big\langle (q)_\infty^{c_2}\left(
1+ \sum\limits_{v} q + O(q^2) \right)\Big\rangle_1\\
& = c_1-c_2 \,.
\end{align*}
Therefore, 
$$\langle \Phi_G(q) \rangle_1=\langle (1-q)\Phi^{\text{TQFT}}_G(q) 
\rangle_1=c_1-c_2-1 \,.
$$

\subsection{The coefficient of $q^2$ in $\Phi_G(q)$}
\lbl{coef.q2}

We need to find the admissible states $(a,b)$ such that 
$\frac{1}{2}(A(a,b)+B(a,b))=2$. Since $A(a,b),B(a,b)\in \mathbb{N}$ 
we have the following cases: 
$$
\begin{array}{|c|c|c|c|c|c|}
\hline
A(a,b) & 4 & 3 & 2 & 1 & 0 \\ \hline
B(a,b) & 0 & 1 & 2 & 3 & 4 \\ \hline
\end{array}
$$

{\bf Case 1}: $(A(a,b), B(a,b))=(4,0)$. If there is a face $p$ such that 
$a_{p}+b_{p} > 0$ then by Corollary \ref{cor.2} we have 
$B(a,b)\geq a_p+b_v\geq a_{p}+b_{p}>0$ where $v$ is a vertex of $p$. Therefore $a_p+b_p=0$ for all faces $p$. 
Similarly, if there exists a face $p$ and a vertex $v\in p$ such that 
$b_v-b_p>0$ then $0=B(a,b)\geq a_p+b_v= a_p+b_p+b_v-b_p\geq b_v-b_p>0$. Therefore 
$a_p+b_v=b_v-b_p=0$ for all $v\in p$. Thus $A(a,b)=4$ is equivalent to
\begin{equation}
\lbl{boundary4}
\sum\limits_{vv'\in p_\infty}b_vb_{v'}=4 \,.
\end{equation}
If $vv'$ is an edge of $G$ and $p$ is a bounded face that contains $vv'$ then we 
have $a_p+b_v=0=a_p+b_{v'}$ and therefore $b_v=b_{v'}$. So by Equation 
\eqref{eq.10} there exists a boundary edge $vv'$ such that $b_{v}=b_{v'}=1$. 
Lemma \ref{lem.triangle} implies that $B(a,b)\geq 2+V_\infty>0$ which is 
impossible. Therefore there are no admissible states $(a,b)$ that 
satisfy $(A(a,b), B(a,b))=(4,0)$.\\

{\bf Case 2}: $(A(a,b), B(a,b))=(3,1)$. If there exists a face $p_0$ such 
that $a_{p_0}+b_{p_0}>0$ then we must have $l(p_0)=3$ and
\begin{itemize}
\item $a_{p_0}+b_{p_0}=1$, $a_p+b_p=0$ for any $p\neq p_0$,
\item $b_vb_{v'}=0$ for all boundary edges $vv'$,
\item $(b_v-b_p)(b_{v'}-b_p)=0$ for all bounded faces $p$ and and edges 
$vv'\in p$.
\end{itemize} 
Let $p\neq p_0$ be a bounded face of $G$. Let $G'$ be the graph obtained 
from $G$ by deleting the boundary edges of $p$ and $(a',b')$ be the 
restriction of $(a,b)$ on $G'$. By part (e) of Lemma \ref{lem.induction}, 
we have 
$A(a',b')=A(a,b)$ and $B(a',b')=B(a,b)$. We can continue this process until 
$G=p_0$. Let $v_0,v_0',v_0''$ be the vertices of $p_0$ then $b_{v_0}b_{v_0'}=0$  
so we can assume that $b_{v_0}=0$. Since $(b_{v_0}-b_{p_0})(b_{v_0'}-b_{p_0})=0$ 
we have $b_{p_0}=0$ and hence $a_{p_0}=a_{p_0}+b_{p_0}=1$. Since 
$1=B(a,b)=a_{p_0}+2(b_{v_0}+b_{v_0'}+b_{v_0''})$ it implies that 
$b_{v_0'}=b_{v_0''}=0$. This gives us the following set of admissible states 
$(a,b)$:
\begin{itemize}
\item $a_{p}=1$ for a triangular face $p$,  $a_{p'}=0$ for $p'\neq p$,
\item  $b_v=0$ for all vertices $v$, 
\end{itemize}
The contribution of this state to $\Phi_G(q)$ is 
$(-1)^1\frac{q^2}{(1-q)^{l(p)}}=-\frac{q^2}{(1-q)^3}=-q^2+O(q^3)$.

On the other hand if $a_p+b_p=0$ for all $p$ then we have
\begin{equation}
\lbl{eq.q2.case2}
\sum\limits_{p}\sum\limits_{vv'\in p}(b_v-b_p)(b_{v'}-b_p)
+  \sum\limits_{vv'\in p_\infty}b_vb_{v'}=3 \,.
\end{equation}
There are at most three 
positive terms in the above equation. If a boundary face $p$ has a boundary edge 
$vv'$ that does not correspond to any positive term then
we have $b_vb_{v'}=(b_v-b_p)(b_{v'}-b_p)=0$ so $b_p=0$ which implies 
that $a_p=0$.
Let $G'$ be the graph obtained from $G$ by deleting the boundary edges of 
$p$ and $(a',b')$ 
be the restriction of $(a,b)$ on $G'$. By part (e) of
Lemma \ref{lem.induction}, we have $A(a',b')=A(a,b)$ and 
$B(a',b')=B(a,b)$. We can continue to do this until all boundary edges  of $G$ are 
$v_iv_i',~ i=1,2,3$. This only 
happens if these three edges together form a triangle. Let us denote the 
triangle's vertices by $v,v',v''$ and let $p,p',p''$
be the bounded faces that contain $vv',v'v'',v''v$ respectively. Note that 
since the positive terms in Equation \eqref{eq.q2.case2} correspond to 
different edges, we must have
\begin{align*}
b_vb_{v'}+(b_v-b_p)(b_{v'}-b_p) &=1\\
b_{v'}b_{v''}+(b_{v'}-b_{p'})(b_{v''}-b_{p'}) &=1\\
b_{v''}b_{v}+(b_{v''}-b_{p''})(b_{v}-b_{p''}) &=1
\end{align*}

{\bf Case 2.1}: If the positive terms are $b_vb_{v'},b_{v'}b_{v''},b_{v''}b_v$
then we must have simultaneously  $b_vb_{v'}=b_{v'}b_{v''}=b_{v''}b_v=1$ and 
$(b_w-b_{\tilde{p}})(b_{w'}-b_{\tilde{p}})=0$ for all faces $\tilde{p}$ and 
edge $ww'$. 
The former implies that $b_v=b_{v'}=b_{v''}=1$. Therefore from Lemma 
\ref{lem.triangle}
we have $B(a,b)\geq 2+3=5$ which is impossible.

{\bf Case 2.2}: If, for instance, $b_vb_{v'}=0$ then we must also have 
$(b_v-b_p)(b_{v'}-b_p)=1$. Thus
we can assume that $b_v=0$ and so $-b_p(b_{v'}-b_p)=1$. This implies that
$b_p=-1$ and $b_{v'}=0$. In particular, we have $b_{v'}b_{v''}=0$ and hence 
$(b_{v'}-b_{p'})(b_{v''}-b_{p'})=1$. 
Since $b_vb_{v''}=0$ we also have  $(b_{v''}-b_{p''})(b_{v}-b_{p''})=1$. In 
particular, this implies that 
$(b_w-b_{\tilde{p}})(b_{w'}-b_{\tilde{p}})=0$ for all faces $\tilde{p}$ and edges 
$ww'\in\tilde{p}$ not on
the boundary. Since $B(a,b)=1$ Lemma \ref{lem.triangle2} implies that we must 
have $b_w=-1$ for all $w\neq v,v',v''$ and $a_p=1$ for all $p\neq p_{\infty}$. 
$$
\includegraphics[scale=1.5]{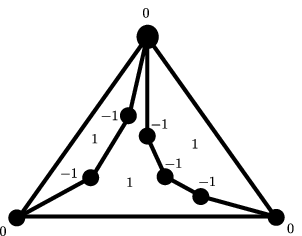}
$$
This corresponds to the following admissible state of $G$: 
\begin{itemize}
\item $a_p=1$ for all bounded faces $p$,
\item $b_v=b_{v'}=b_{v''}=0$ where $v,v',v''$ are the vertices of a 3-cycle in 
$G$,
\item $b_w=-1$ for all vertices $w$ inside the 3-circle mentioned above,
\item $b_{\tilde{w}}=0$ for any other vertex $w$.
\end{itemize}

The contribution of this state to $\Phi_G(q)$ is 
$$(-1)^1\frac{q^2}{(1-q)^{\text{deg}_\Delta(v)+\text{deg}_\Delta(v')+\text{deg}_\Delta(v'')-3}}=-q^2+O(q^3)$$
where $\text{deg}_\Delta(v)$ is the degree of $v$ in the triangle $\Delta=vv'v''$.

{\bf Case 3}: We consider the two cases $(A(a,b), B(a,b)) 
=(2,2)$ and $(A(a,b), B(a,b)) =(1,3)$ together. Since $A(a,b)\leq 2$ we should have 
$a_p+b_p=0$ for all faces $p$ and $A(a,b)=2$ is equivalent to
 $$
\sum\limits_{p}\sum\limits_{vv'\in p}(b_v-b_p)(b_{v'}-b_p)
+  \sum\limits_{vv'\in p_\infty}b_vb_{v'}=2
$$
There are at most two positive terms in the above equation. If a boundary face $p$ has a boundary edge 
$vv'$ that does not correspond to any positive term then
we have $b_vb_{v'}=(b_v-b_p)(b_{v'}-b_p)=0$ so $b_p=0$ which implies 
that $a_p=0$. By part (d)
of Lemma
\ref{lem.induction}, it follows that if $w$ is a boundary vertex of $p$ 
then 
$B(a,b)\geq 2b_w$ and since $B(a,b)\leq 3$ we have $b_w=0$ or $1$. Therefore 
by parts (b,c) of Lemma \ref{lem.induction} we can remove the boundary 
edges of $p$ to obtain a new graph $G'$ that satisfies $A(a,b)=A'(a,b)$ 
and $B(a,b)=B'(a,b)$ or  $B(a,b)=B'(a,b)+1$ where $A'(a,b),B'(a,b)$ are 
the restrictions of $A(a,b)$ and $B(a,b)$ on $G'$. By continuing this 
process until $G=\emptyset$, it is easy to see that  we must have 
$A(a,b)=0$, $B(a,b)\leq 1$ and $B(a,b)=1$ if and only if there exists a 
unique boundary vertex $w$ of $p$ such that $b_w=1$. Thus there are no 
admissible states that satisfy  $(A(a,b), B(a,b))=(2,2)$ or 
$(A(a,b), B(a,b))=(1,3)$.

{\bf Case 4}:  $(A(a,b), B(a,b))=(0,4)$. Since $A(a,b)=0$, we should have 
\begin{align}
a_p+b_p &= 0 ~\text{for all faces $p$} \lbl{eq.41} \\
(b_v-b_p)(b_{v'}-b_p)&= 0 ~\text{for all faces $p$ and edges $vv'\in p$} 
\lbl{eq.42} \\
b_vb_{v'}&= 0 ~\text{for all edges $vv'\in p$} \lbl{eq.43}
\end{align}
Let $p$ be a boundary face of $G$ and $vv'\in p$ be a boundary edge. Equations \eqref{eq.42} and \eqref{eq.43} imply that $b_p=0$ and so $a_p=0$ by Equation \eqref{eq.41}. Let $G'$ be 
the 
graph obtained from $G$ by deleting the boundary edges of $G$ and $(a',b')$ 
be the restriction of $(a,b)$ on $G'$. By part (e) of
Lemma \ref{lem.induction} we 
have $A(a',b')=A(a,b)$, $B(a',b')=B(a,b)-2n_p$ where $n_p\in \mathbb{N}$. 
Since $B(a,b)=4$ we have $n_p\leq 2$ and 
\begin{itemize}
\item $n_p=2$ if and only if there exist 
either  exactly two boundary vertices $v,w \in p$ that are not connected by 
an edge such that $b_v=b_{v'}=1$ or exactly one boundary vertex $v\in p$ such 
that $b_v=2$ and $b_{v'}=0$ for all other boundary vertices $v'\in p$
\item $n_p=1$ if and only if there exists exactly one boundary vertex 
$v\in p$ such that $b_v=1$ and $b_v'=0$ for any other boundary vertex $v'$ 
of $p$.
\end{itemize} 
Similarly,  by continuing this process it is easy to show that an admissible 
state $(a,b)$ such that $(A(a,b), B(a,b))=(0,4)$ must satisfy one the 
following.
\begin{itemize}
\item $b_v=b_{v'}=1$ for a pair of vertices that are not connected by an 
edge of $G$, $b_w=0$ for any other vertex $w$,
\item $a_p=0$ for all faces $p$.
\end{itemize}
The contribution of this state to $\Phi_G(q)$ is 
$\frac{q^2}{(1-q)^{\text{deg}(v)+\text{deg}(v')}}
=-q^2+O(q^3)$.
\begin{itemize}
\item $b_v=2$ for a vertex $v$, $b_w=0$ for any other vertex $w$,
\item $a_p=0$ for all faces $p$.
\end{itemize}
The contribution of this state to $\Phi_G(q)$ is 
$\frac{q^2}{(1-q)_2^{\text{deg}(v)}}=
-q^2+O(q^3)$.

It follows from Theorem \ref{thm.phi0}, Section \ref{coef.q} and cases 1-4 
that
\begin{align*}
\langle{\Phi^{\text{TQFT}}_G(q)}\rangle_{2} &= \langle{(q)_\infty^{c_2}(1+ \sum\limits_{v}\frac{q}{(1-q)^{\text{deg}(v)}}+(-c_3+c_1+\frac{c_1(c_1-1)}{2}-c_2))q^2}\rangle_{2}\\ \notag
& =  \langle{(q)_\infty^{c_2}(1+ q(c_1+2c_2 q)+(\frac{c_1(c_1+1)}{2}-c_2-c_3)q^2}\rangle_{2}\\  \notag
&= \langle{(1-c_2 q+\frac{c_2(c_2-3)}{2}q^2)(1+ c_1 q+(\frac{c_1(c_1+1)}{2}+c_2-c_3)q^2}\rangle_{2}\\  \notag
&= \frac{(c_1 -c_2)^2}{2} -c_3+\frac{c_1-c_2}{2} \,.
\end{align*}
Therefore, 
\begin{align*}
\langle \Phi_G(q) \rangle_2 &=\langle (1-q)\Phi^{\text{TQFT}}_G(q) 
\rangle_2\\
&=\Big\langle (1-q)(1+(c_1-c_2)q+(\frac{(c_1 -c_2)^2}{2} -c_3 + \frac{c_1-c_2}{2})q^2) \Big\rangle_2\\
&=\frac{1}{2}((c_1 -c_2)^2-2c_3-c_1+c_2) \,.
\end{align*} 
This completes the proof of Theorem \ref{thm.1}.
\qed

\subsection{Proof of Lemma \ref{lem.sum}}
\lbl{sub.lem.sum}

Fix a planar graph $G$ and consider $G \cdot P_r$ where
$P_r$ is a polygon with $r$ sides and vertices $b_1,\dots, b_r$
as in the following figure 
$$
\psdraw{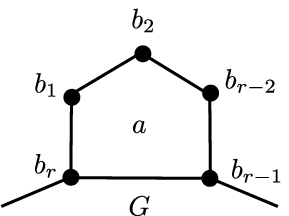}{2in}
$$
Consider the corresponding portion $S(b_{r-1},b_r)$
of the formula of $\Phi_{G \cdot P_r}(q)$
\begin{equation}
\lbl{eq.Sbb}
S(b_{r-1},b_r)=\sum_{a,b_1,\dots,b_{r-2}} (-1)^{r a} 
\frac{q^{\frac{r}{2} a^2 + a(b_1 + \dots b_r) + \sum_{i=1}^{r-2} b_i b_{i+1}
+ b_1 b_r + \sum_{i=1}^{r-2} b_i + \frac{r-2}{2} a}}{
(q)_{b_1} (q)_{b_2} \dots (q)_{b_{r-2}} (q)_{b_1+a} (q)_{b_2+a} \dots
(q)_{b_r+a}}
\end{equation}
for fixed $b_{r-1},b_r \geq 0$. Armond-Dasbach \cite[Thm3.7]{AD} and
Andrews \cite{Andrews} prove that 
$$
S(b_{r-1},0)= (q)_\infty^{-r+1} h_r(q) 
$$
for all $b_{r-1} \geq 0$. Summing over the remaining variables in the
formula for $\Phi_{G \cdot P_r}(q)$ concludes the proof of the Lemma.
\qed


\section{The computation of $\Phi_G(q)$}
\lbl{sec.compute}

\subsection{The computation of $\Phi_{L8a7}(q)$ in detail}
\lbl{sub.L8a7}

In this section we explain in detail the computation of 
$\Phi_{L8a7}(q)$. Consider the planar graph of the alternating link $L8a7$
shown in Figure \ref{f.L8a7}, with the marking of its vertices by
$b_i$ for $i=1,\dots,6$ and its bounded faces by $a_j$ for $j=1,2,3$.

\begin{figure}[htpb]
$$\psdraw{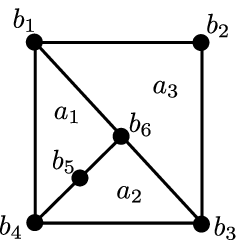}{1.2in}$$
\caption{The planar graph of the link $L8a7$.}
\lbl{f.L8a7}
\end{figure}
           
Consider the minimum values of the $b$-variables at each bounded
face: 
\begin{align*}
\bar{b}_1 & =\min\{b_1,b_4,b_5,b_6\} \\ 
\bar{b}_2 & =\min\{b_3,b_4,b_5,b_6\} \\
\bar{b}_3 & =\min\{b_1,b_2,b_3,b_6\} \,. 
\end{align*}

We have
\begin{align}
\lbl{dist.Q2}
\frac{1}{2}A(a,b)& = 2(a_1+\bar{b}_1)^2+(a_1+\bar{b}_1)(b_1+b_4+b_5+b_6-4\bar{b}_1)\nonumber \\
& +2(a_2+\bar{b}_2)^2 +(a_1+\bar{b}_2)(b_3+b_4+b_5+b_6-4\bar{b}_2) \nonumber \\
& +2(a_3+\bar{b}_3)^2 +(a_3+\bar{b}_3)(b_1+b_2+b_3+b_6-4\bar{b}_3) \nonumber \\
& +\frac{1}{2}(b_1-\bar{b}_1)(b_6-\bar{b}_1)+(b_6-\bar{b}_1)(b_5-\bar{b}_1)+(b_5-\bar{b}_1)(b_4-\bar{b}_1)+(b_4-\bar{b}_1)(b_1-\bar{b}_1) \nonumber \\
& +\frac{1}{2}(b_3-\bar{b}_2)(b_4-\bar{b}_2)+(b_4-\bar{b}_2)(b_5-\bar{b}_2)+(b_5-\bar{b}_2)(b_6-\bar{b}_2)+(b_6-\bar{b}_2)(b_3-\bar{b}_2) \nonumber \\
& + \frac{1}{2}(b_1-\bar{b}_3)(b_2-\bar{b}_3)+(b_2-\bar{b}_3)(b_3-\bar{b}_3)+(b_3-\bar{b}_3)(b_6-\bar{b}_3)+(b_6-\bar{b}_3)(b_1-\bar{b}_3) \nonumber \\
& + \frac{1}{2}(b_1b_2+b_2b_3+b_3b_4+b_4b_1)\nonumber \\
& = C(a_1,a_2,a_3,b_1,b_2,b_3,b_4,b_5,b_6)+D(b_1,b_2,b_3,b_4,b_5,b_6)
\end{align}
and
\begin{align}
\notag\frac{1}{2}B(a,b)&= a_1+a_2+a_3+b_1+b_2+b_3+b_4+b_5+b_6\\
\lbl{dist.L}
&= \frac{a_1+b_1}{2}+\frac{a_1+b_5}{2}+\frac{a_2+b_5}{2}+\frac{a_2+b_6}{2}+\frac{a_3+b_1}{2}+\frac{a_3+b_6}{2}+b_2+b_3+b_4 \,. 
\end{align}
If $\frac{1}{2}(A(a,b)+B(a,b))\leq N$ then $\frac{1}{2}B(a,b)\leq N$, so 
\begin{align}
 0 \leq & b_2\leq N \lbl{b2}\\
 0 \leq & b_3\leq N-b_2 \lbl{b3}\\
 0 \leq & b_4\leq N-b_2-b_3 \,. \lbl{b4}
\end{align}
Let us set 
\begin{align}
\lbl{b1}
b_1 & =0 \,.
\end{align}
Equation \eqref{dist.L} implies that 
$0\leq \frac{a_1+b_1}{2}\leq N-b_2-b_3-b_4$ which implies that 
$0\leq a_1\leq 2(N-b_2-b_3-b_4)$. It follows from 
$0\leq \frac{a_1+b_5}{2}\leq N$ that
\begin{equation}
\lbl{b5}
-2(N-b_2-b_3-b_4)\leq b_5\leq 2(N-b_2-b_3-b_4) \,.
\end{equation}
Since $0\leq \frac{a_2+b_5}{2}\leq N-b_2-b_3-b_4$ from \eqref{b5} we have $-2(N-b_2-b_3-b_4)\leq a_2\leq 4(N-b_2-b_3-b_4)$. Therefore, since $0\leq a_2\leq \frac{a_2+b_6}{2}$ we have
\begin{equation}
\lbl{b6}
-4(N-b_2-b_3-b_4)\leq b_6\leq 4(N-b_2-b_3-b_4) \,.
\end{equation}
Equations \eqref{b2}-\eqref{b6} in particular bound $b_2, b_3, b_4, b_5$ and 
$b_6$ from above and from below by linear forms in $N$. But even better,
Equations \eqref{b2}-\eqref{b6} allow for an iterated summation 
for the $b_i$ variables which improves the computation of the 
$\Phi_{L8a7}(q)$ series.

To bound $a_1,a_2,a_3$ we will use the auxiliary function
$$
u(c,d)=\Big{[}\frac{-c+\sqrt{c^2+2d}}{2}\Big{]}
$$
where the integer part $[x]$ of a real number $x$ is the biggest 
integer less than or equal to $x$. The argument of $u(c,d)$
inside the integer part is one of the solutions to the equation
$2x^2+cx-d=0$. Let
\begin{align*}
\tilde{b}_1 &= b_1+b_4+b_5+b_6-4\bar{b}_1 \\
\tilde{b}_2 &= b_3+b_4+b_5+b_6-4\bar{b}_2 \\
\tilde{b}_3 &= b_1+b_2+b_3+b_6-4\bar{b}_3 \\
\tilde{D} &= D(b_1,b_2,b_3,b_4,b_5,b_6) + b_2 + b_3 + b_4
\end{align*}
Since 
$$
2(a_1+\bar{b}_1)^2+(a_1+\bar{b}_1)\tilde{b}_1\leq 
N-\tilde{D}
$$
we have
\begin{equation}
\lbl{a1}
-\bar{b}_1\leq a_1 \leq -\bar{b}_1+u(\tilde{b}_1,N-\tilde{D})
\end{equation}
where the left inequality follows from the fact that $a_1\geq -b_i, i=1,4,5,6$.
Similarly we have
\begin{equation}
 \lbl{a2}
-\bar{b}_2  \leq  a_2  \leq -\bar{b}_2+u(\tilde{b}_1,N-\tilde{D}-2(a_1+\bar{b}_1)^2-(a_1+\bar{b}_1)\tilde{b}_1)
\end{equation}
and
\begin{equation}
 \lbl{a3}
-\bar{b}_3  \leq  a_3  \leq -\bar{b}_3 +u(\tilde{b}_1,N-\tilde{D}-2(a_1+\bar{b}_1)^2-(a_1+\bar{b}_1)\tilde{b}_1-2(a_2+\bar{b}_2)^2-(a_2+\bar{b}_2)\tilde{b}_2)
\end{equation}
Note that Equations \eqref{a1}-\eqref{a3} allow for an iterated summation
in the $a_i$ variables, and in particular imply that the span of the 
$a_i$ variables is bounded by a linear form of $\sqrt{N}$.

It follows that 
\begin{align*}
\Phi_{L8a7}(q) + O(q)^{N+1} &=
(q)_\infty^{8}\sum\limits_{(a,b)} 
\frac{q^{\frac{1}{2}(A(a,b)+B(a,b))}}{(q)_{a_1+b_1}(q)_{a_1+b_4}(q)_{a_1+b_5}(q)_{a_1+b_6}(q)_{a_2+b_3}(q)_{a_2+b_4}(q)_{a_2+b_5}
(q)_{a_2+b_6}} \\
& \cdot \frac{1}{
(q)_{a_3+b_1}(q)_{a_3+b_2}(q)_{a_3+b_3}(q)_{a_3+b_6}(q)_{b_1} (q)_{b_2}
(q)_{b_3} (q)_{b_4}}  + O(q)^{N+1} 
\end{align*}
\noindent
where $(a,b)=(a_1,a_2,a_3,b_1,b_2,b_3,b_4,b_5,b_6)$ 
satisfy the inequalities \eqref{b2}-\eqref{b6} and \eqref{a1}-\eqref{a3}. 
We give the first 21 terms of this series in the Table \ref{f.Phi.values}.

\subsection{The computation of $\Phi_G(q)$ by iterated summation}
\lbl{sub.comment}

Our method of computation requires not only the planar graph with its
vertices and faces (which is relatively easy to automate), but also 
the inequalities for the $b_i$ and $a_j$ variables which lead to an iterated
summation formula for $\Phi_G(q)$. Although Theorem \ref{thm.2} implies the 
existence of an iterated summation formula for every planar graph, we did not
implement this algorithm in general.

Instead, for each of the 11 
graphs that appear in Figures \ref{f.graphs67} and \ref{f.graphs8},
we computed the corresponding inequalities for the iterated summation by hand. 
These inequalities are too long to present them here, but we have them 
available. A consistency check of our computation is obtained by Equation
\eqref{eq.Phiconv}, where the shifted colored Jones polynomial of an
alternating link is available from \cite{B-N} for several values. Our data
matches those values.

\subsection*{Acknowledgment}
The first author wishes to thank Don Zagier for a generous sharing of his
time and his ideas and S. Zwegers for enlightening conversations. The second
author wishes to thank Chun-Hung Liu for conversations on combinatorics of
plannar graphs. The results of this project were presented by the first 
author in the Arbeitstagung in Bonn 2011, in the Spring School in Quantum
Geometry in Diablerets 2011, in the Clay Research Conference in Oxford 2012
and the Low dimensional Topology and Number Theory, Oberwolfach 2012.
We wish to thank the organizers for their invitation and hospitality.


\appendix

\section{Tables}
\lbl{sec.table}


In this section we give various tables of graphs, and their corresponding
alternating knots (following Rolfsen's notation \cite{Rf}) and links
(following Thistlethwaite's notation \cite{B-N}) and several terms of
$\Phi_G(q)$. In view of an expected positive answer to Question \ref{que.1},
we will list {\em irreducible} graphs, i.e., simple planar 2-connected graphs
which are not of the form $G_1 \cdot G_2$ (for the operation $\cdot$
defined in Section \ref{sub.properties}). 

\begin{itemize}
\item
The first table gives number of alternating links with at most 10 crossings and 
the number of irreducible graphs with at most 10 edges
\begin{equation}
\lbl{tab.edge}
\begin{array}{|c|c|c|c|c|c|c|c|c|} \hline 
\text{crossings}=\text{edges} & 3 & 4 & 5 & 6 & 7 & 8 & 9 & 10 \\ \hline
\text{alternating links} & 1 & 2 & 3 & 8 & 14 & 39 & 96 & 297 \\ \hline
\text{irreducible graphs} & 1 & 1 & 1 & 3 & 3 & 8 & 17 & 41 \\ \hline
\end{array}
\end{equation}
To list planar graphs, observe that they are {\em sparse}: if $G$ is a planar
graph which is not a tree, with $V$ vertices and $E$ edges then
$$
V \leq E \leq 3V-6 \,.
$$
\item
The next table gives the number of planar 2-connected irreducible graphs 
with at most 9 vertices 
\begin{equation}
\lbl{tab.vertex}
\begin{array}{|c|c|c|c|c|c|c|c|} \hline 
\text{vertices} & 3 & 4 & 5 & 6 & 7 & 8 & 9  \\ \hline
\text{graphs} & 1 & 2 & 5 & 19 & 106 & 897 & 10160  \\ \hline
\end{array}
\end{equation}
\item
Tables \ref{f.graphs345}, \ref{f.graphs67}, \ref{f.graphs8} and 
\ref{f.graphs9} give the list of irreducible graphs with at most 9 edges. 
These tables were constructed by listing all
graphs with $n \leq 9$ vertices, selecting those which are planar, 
and further selecting those that are irreducible. Note that if
$G$ is a planar graph with $E \leq 9$ edges, $V$ vertices and $F$ faces then 
$E-V=F-2 \geq 0$ hence $V \leq E \leq 9$.
\item 
Tables \ref{f.tait.knots8} and  \ref{f.tait.links8} give the 
reduced Tait graphs of all alternating knots and links (and their mirrors)
with at most 8 crossings. Here $P_r$ is the planar polygon with
$r$ sides and $-K$ denotes the mirror of $K$. Moreover, 
the notation $G=G_1\cdot G_2\cdot G_3$ indicates that  
$\Phi_G(q)=\Phi_{G_1}(q)\Phi_{G_2}(q)\Phi_{G_3}(q)$ by Lemma \ref{lem.sum}.
\item
Table \ref{f.tait.edge8} gives the alternating knots and links with at most
8 crossings for the irreducible graphs with at most 8 edges.
\item
Table \ref{f.Phi.values} gives the first 21 terms of of $\Phi_G(q)$ for 
all irreducible graphs with at most 8 edges. Many more terms are 
available from
\begin{center}
{\tt \url{http://www.math.gatech.edu/~stavros/publications/phi0.graphs.data/}}
\end{center}
\end{itemize}

\begin{figure}[htpb]
$$
\psdraw{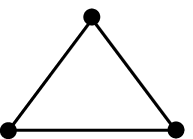}{0.7in} \qquad \psdraw{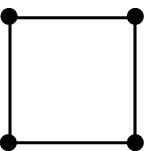}{0.55in} \qquad
\psdraw{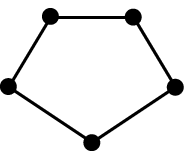}{0.7in}  
$$
\caption{The irreducible planar graphs $G_0^3,G_0^4$ 
and $G_0^5$ with 3, 4 and 5 edges.}
\lbl{f.graphs345}
\end{figure}

\begin{figure}[htpb]
$$
\psdraw{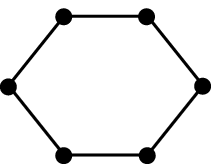}{0.7in} \qquad  \psdraw{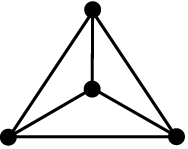}{0.7in} \qquad \psdraw{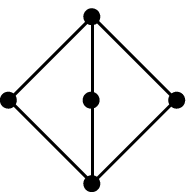}{0.7in} 
$$
$$
\psdraw{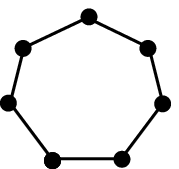}{0.7in} \qquad  \psdraw{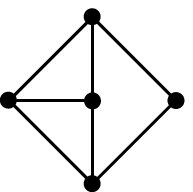}{0.7in} \qquad \psdraw{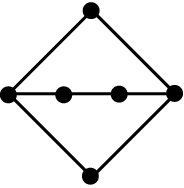}{0.7in}  
$$
\caption{The irreducible planar graphs with 6 and 7 edges: $G_0^6,G_1^6,G_2^6$ on the top and $G_0^7,G_1^7,G_2^7$ on the bottom.}
\lbl{f.graphs67}
\end{figure}

\begin{figure}[htpb]
$$
\psdraw{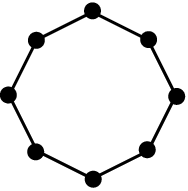}{0.7in} \qquad \psdraw{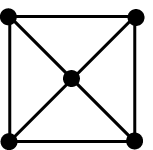}{0.7in} \qquad \psdraw{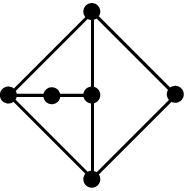}{0.7in} \qquad
\psdraw{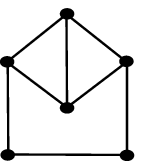}{0.7in} 
$$
$$ 
 \psdraw{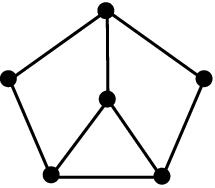}{0.7in}  \qquad \psdraw{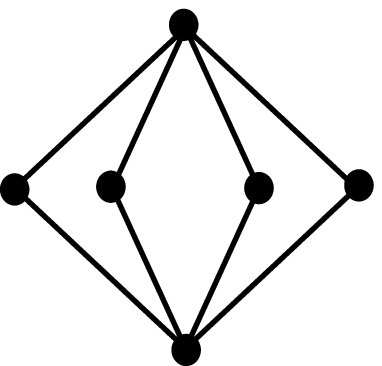}{0.8in} \qquad \psdraw{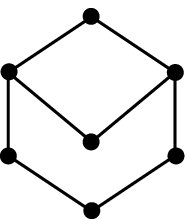}{0.7in} \qquad
\psdraw{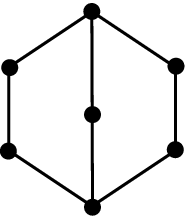}{0.7in} 
$$
\caption{The irreducible planar graphs with 8 edges: $G_0^8,\dots,G_3^8$ on the top (from left to right) and $G_4^8,\dots,G_7^8$ on the bottom.}
\lbl{f.graphs8}
\end{figure}

\begin{figure}[htpb]
$$
 \psdraw{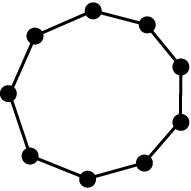}{0.7in} \quad
\psdraw{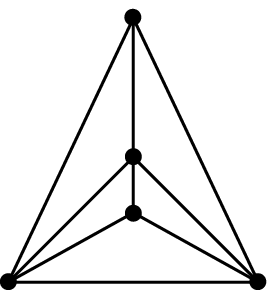}{0.7in} \quad 
\psdraw{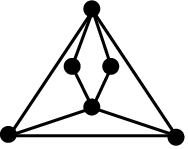}{0.7in} \quad \psdraw{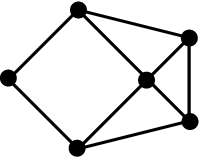}{0.7in} \quad 
\psdraw{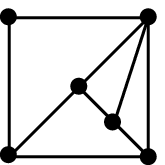}{0.7in} \quad \psdraw{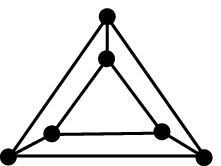}{0.8in}  
$$
$$ 
\psdraw{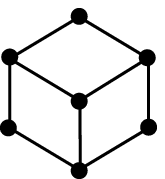}{0.7in}  \quad
  \psdraw{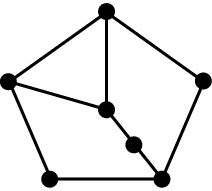}{0.7in}  \quad 
\psdraw{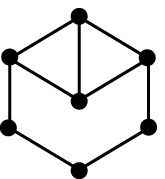}{0.7in} \quad \psdraw{911}{0.8in} \quad
 \psdraw{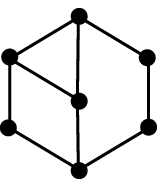}{0.7in} \quad \psdraw{913}{0.7in} 
$$
$$
 \psdraw{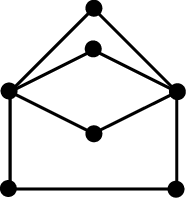}{0.7in} \quad
 \psdraw{919}{0.7in} \quad \psdraw{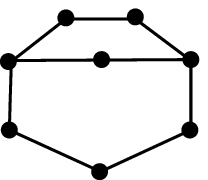}{0.7in} \quad \psdraw{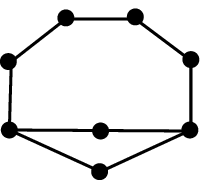}{0.7in} \quad 
\psdraw{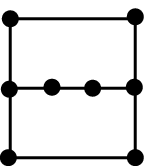}{0.7in} 
$$
\caption{The irreducible planar graphs with 9 edges: $G_0^9,\dots,G_5^9$ on the top, $G_6^9,\dots,G_{11}^9$ on the middle and $G_{12}^{9},\dots,G_{16}^9$ on the bottom.}
\lbl{f.graphs9}
\end{figure}

\begin{figure}[htpb]
{\small
$$
\begin{array}{|c|c|c||c|c|c||c|c|c||c|c|c|} \hline
K &  G & -G & K & G & -G & K & G & -G & K & G & -G \\ \hline
0_1  & P_2  &  P_2 & 7_2  &   P_6 &  P_3 & 8_4 & P_3 & P_4\ccdot P_5 & 8_{13} & P_3\ccdot P_3\ccdot P_4 & P_3\ccdot P_3 \\
3_1 &  P_3  &  P_2 & 7_3  &  P_5 &  P_4  & 8_5 & G^8_7 & P_3 & 8_{14} & P_3\ccdot P_4 & P_3\ccdot P_3 \ccdot P_3 \\ 
4_1  &  P_3  & P_3 & 7_4  &   P_4\ccdot P_4 &  P_3 & 8_6 & P_3\ccdot P_4 & P_5 & 8_{15} & P_3\ccdot P_3\ccdot P_3 & G^6_2 \\
5_1  &  P_5 &  P_2 & 7_5  &  P_3\ccdot P_4 & P_4 & 8_7 & P_3\ccdot P_5 & P_4 & 8_{16} & G^8_4 & G^6_1 \\
5_2  &  P_4  &  P_3 & 7_6  &  P_3\ccdot P_4 &  P_3\ccdot P_3 & 8_8 & P_3\ccdot P_5 & P_3\ccdot P_3 & 8_{17} & G^7_1 & G^7_1 \\ 
6_1  &   P_5 & P_3  & 7_7 & P_3\ccdot P_3\ccdot P_3  & P_3\ccdot P_3  & 8_9 & P_3\ccdot P_4 & P_3\ccdot P_4 & 8_{18} & G^8_1 & G^8_1 \\
6_2  &  P_3\ccdot P_4 & P_3  & 8_1  & P_7  & P_3 & 8_{10} & G^7_2 & P_3\ccdot P_3 & & & \\
6_3  &  P_3\ccdot P_3 & P_3\ccdot P_3  & 8_2  & P_3\ccdot P_6  & P_3 & 8_{11} & P_3\ccdot P_4 & P_3\ccdot P_4 & & & \\
7_1  &  P_7 &   P_2  & 8_3  & P_5  & P_5 & 8_{12} & P_3\ccdot P_4 & P_3\ccdot P_4 & & & \\ \hline
\end{array}
$$
}
\caption{The reduced Tait graphs of the alternating knots with at most 
8 crossings}
\lbl{f.tait.knots8}
\end{figure}

\begin{figure}[htpb]
{\small
$$
\begin{array}{|c|c|c||c|c|c||c|c|c||c|c|c|} \hline
L &  G & -G & L & G & -G & L & G & -G & L & G & -G \\ \hline
2a1  & P_2  &  P_2 & 7a2  &   P_3\ccdot P_3 &  G^6_2 & 8a4 & P_3\ccdot P_4 & P_3\ccdot P_3\ccdot P_3 & 8a13 & P_4\ccdot P_4 & P_4 \\
4a1 &  P_4  &  P_2 & 7a3  &  G^7_2 &  P_3  & 8a5 & P_4 & P_3\ccdot P_3\ccdot P_4 & 8a14 & P_8 & P_2 \\ 
5a1  &  P_3\ccdot P_3  & P_3 & 7a4  &   P_5 &  P_3\ccdot P_3 & 8a6 & P_6 & P_3\ccdot P_3 & 8a15 & P_5 & P_3\ccdot P_3\ccdot P_3 \\
6a1  &  P_4 &  P_3\ccdot P_3 & 7a5  &  P_3\ccdot P_4 & P_3\ccdot P_3 & 8a7 & G^8_2 & G^6_1 & 8a16 & G^8_3 & G^6_1 \\
6a2  &  P_4  &  P_4 & 7a6  &  P_3\ccdot P_5 &  P_3 & 8a8 & P_3\ccdot P_4\ccdot P_3 & P_3\ccdot P_3 & 8a17 &  P_3\ccdot P_4 & G^6_2 \\ 
6a3  &   P_6 & P_2  & 7a7 & P_4  & G^6_2  & 8a9 & P_3\ccdot P_3\ccdot P_3 & P_3\ccdot P_3\ccdot P_3 & 8a18 & G^8_6 & P_3 \\
6a4  &  G^6_1 & G^6_1  & 8a1  & G^7_1  & P_3\ccdot G^6_1 & 8a10 & P_3\ccdot P_4 & P_3\ccdot P_3 & 8a19 & G^7_1 & G^7_1 \\
6a5  &  P_3 & G^6_2  & 8a2  & P_3\ccdot P_3  & P_3\ccdot G^6_2 & 8a11 & P_3\ccdot P_5 & P_4 & 8a20 & G^6_2 & G^6_2 \\
7a1  &  G^7_1 & G^6_1 & 8a3  & G^7_2  & P_3\ccdot P_3 & 8a12 & P_6 & P_4 & 8a21 & P_4 & G^8_5 \\ \hline
\end{array}
$$
}
\caption{The reduced Tait graphs of the alternating links with at most 
8 crossings}
\lbl{f.tait.links8}
\end{figure}

\begin{figure}[htpb]
{\small
$$
\begin{array}{|l|l|} \hline
G^6_1 & L6a4 \quad -\!\!L6a4 \quad -\!\!L7a1 \quad -\!\!L8a7 
\quad -\!\!8_{16} \quad -\!\!L8a16 
\\ \hline 
G^6_2 & -\!\!L6a5 \quad -L7a2 \quad -L7a7 \quad -L8a17 \quad -8_{15}  \quad L8a20 
\quad -\!\!L8a20 
\\ \hline
G^7_1 & L7a1  \quad L8a1 \quad 8_{17} \quad -8_{17} \quad L8a19 \quad -L8a19
\\ \hline
G^7_2 & 8_{10}  \quad L7a3 \quad L8a3
\\ \hline
G^8_1 & 8_{18} \quad -\!\!8_{18} 
\\ \hline
G^8_2 & L8a7 
\\ \hline 
G^8_3 & L8a16 
\\ \hline
G^8_4 & 8_{16} 
\\ \hline 
G^8_5 &  -\!\!L8a21 
\\ \hline
G^8_6 & L8a18 
\\ \hline
G^8_7 & 8_5 
\\ \hline
\end{array}
$$
}
\caption{The irreducible planar graphs with at most 8 edges and the
corresponding alternating links}
\lbl{f.tait.edge8}
\end{figure}


\begin{figure}[htpb]
{\small
$$
\begin{array}{|l|l|} \hline
G & \Phi_G(q) +O(q)^{21} \\ \hline\hline
G^6_1 
& 1 - 3 q - q^2 + 5 q^3 + 3 q^4 + 3 q^5 - 7 q^6 - 5 q^7 - 8 q^8 - 6 q^9 + 
 6 q^{10} 
\\ & 
+ 7 q^{11} + 12 q^{12} + 15 q^{13} + 16 q^{14} - 3 q^{15} - q^{16} 
- 15 q^{17} -  21 q^{18} - 31 q^{19} - 30 q^{20}
\\ \hline
G^6_2 & 1 - 2q + q^2 + 3q^3 - 2q^4 - 2q^5 - 3q^6 + 3q^7 + 4q^8 + q^9 + 
 3q^{10} 
\\ &
- 6q^{11} - 5q^{12} - 3q^{13} + q^{15} + 7q^{16} + 9q^{17} + 3q^{18} - 6q^{20}
\\ \hline\hline
G^7_1 & 1 - 3q + q^2 + 5q^3 - 3q^4 - 3q^5 - 6q^6 + 6q^7 + 8q^8 + 3q^9 + 6q^{10} 
\\ &
- 13q^{11} - 14q^{12} - 9q^{13} - q^{14} + 3q^{15} + 21q^{16} + 27q^{17} + 
 14q^{18} + 3q^{19} - 17q^{20}
\\ \hline
G^7_2 & 1 - 2q + q^2 + q^3 - 3q^4 + q^5 + q^6 + 3q^7 - 2q^8 - 4q^9 + q^{10} 
\\ &
+ 4q^{12} + 5q^{13} - 2q^{14} - 5q^{15} - 4q^{16} - 2q^{17} - 2q^{18} 
+ 5q^{19} +  8q^{20}
\\ \hline\hline
G^8_1 & 1 - 4q + 2q^2 + 9q^3 - 5q^4 - 8q^5 - 14q^6 + 10q^7 + 21q^8 + 14q^9 + 19q^{10} 
\\ &
- 29q^{11} - 42q^{12} - 42q^{13} - 20q^{14} + 3q^{15} + 64q^{16} + 
 104q^{17} + 88q^{18} + 55q^{19} - 25q^{20} 
 \\ \hline
G^8_2 & 1 - 3 q + 3 q^2 + 4 q^3 - 8 q^4 - 2 q^5 + 2 q^6 + 12 q^7 + 3 q^8 
- 15 q^9 -  4 q^{10} 
\\ & - 14 q^{11} + 10 q^{12} + 25 q^{13} + 15 q^{14} - 18 q^{16} - 22 q^{17} 
- 39 q^{18} - 12 q^{19} + 19 q^{20}
\\ \hline
G^8_3 & 1 - 3q + q^2 + 3q^3 - 3q^4 + 3q^5 + 4q^7 - 6q^8 - 10q^9 + q^{10} 
\\ & 
- q^{11} + 9q^{12} + 13q^{13} + 3q^{14} - 9q^{15} - 3q^{16} - 6q^{17} 
- 4q^{18} +  5q^{19} + 13q^{20} 
\\ \hline
G^8_4 & 1 - 3 q + 2 q^2 + 3 q^3 - 6 q^4 + q^5 + 2 q^6 + 8 q^7 - 3 q^8 
- 13 q^9 
\\ & 
- 3 q^{11} + 13 q^{12} + 19 q^{13} + q^{14} - 15 q^{15} - 20 q^{16} 
- 16 q^{17} - 13 q^{18} +  15 q^{19} + 37 q^{20}
\\ \hline
G^8_5 & 1 - 3 q + 3 q^2 + 5 q^3 - 8 q^4 - 5 q^5 - q^6 + 15 q^7 + 12 q^8 
- 8 q^9 -  7 q^{10}
\\ & - 31 q^{11} - 11 q^{12} + 14 q^{13} + 30 q^{14} + 35 q^{15} 
+ 27 q^{16} +  8 q^{17} - 48 q^{18} - 66 q^{19} - 72 q^{20}
\\ \hline
G^8_6 & 1 - 2q + q^2 + q^3 - q^4 + 2q^5 - 2q^6 - q^7 - 2q^8 + 2q^9 + 5q^{10}
\\ & 
- q^{11} - q^{12} - 3q^{13} - 2q^{14} + 5q^{16} - 2q^{18} - q^{19} - q^{20}
\\ \hline
G^8_7 & 1 - 2q + q^2 - 2q^4 + 3q^5 - 3q^8 + q^9 + 4q^{10}
\\ &
- q^{11} - 2q^{12} - 2q^{13} - 3q^{14} + 3q^{15} + 7q^{16} + 2q^{17} 
- 4q^{18} - 4q^{19} - 4q^{20}
\\ \hline
\end{array}
$$
}
\caption{The first 21 terms of $\Phi_G(q)$ for the irreducible planar
graphs with at most 8 edges}
\lbl{f.Phi.values}
\end{figure}

\begin{figure}[htpb]
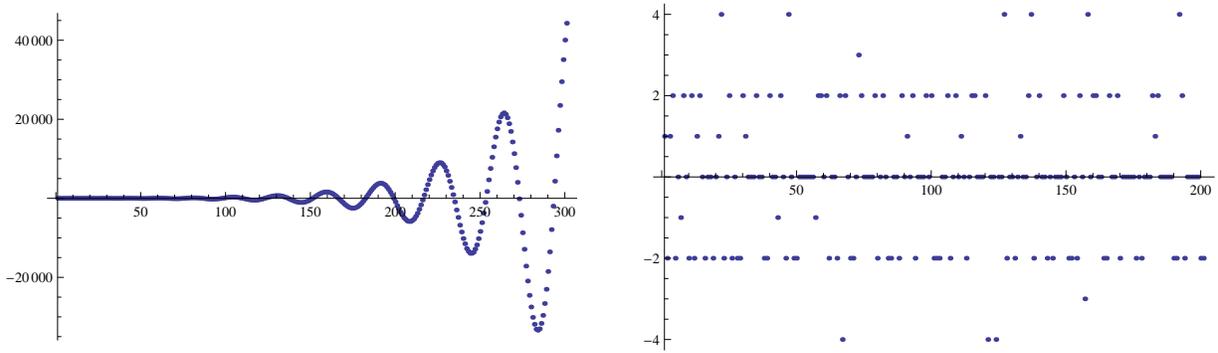

$$
\psdraw{{phi.G6.2}}{3.0in} \qquad \psdraw{{phi.polygon.G6.2}}{3.0in} 
$$
\caption{Plot of the coefficients of $\Phi_{G^6_2}(q)$ on the left and 
$h_4(q)^2$ (keeping in mind that $G^6_2$ has two bounded square faces)
on the right.}
\lbl{f.phi.plot}
\end{figure}


\bibliographystyle{hamsalpha}
\bibliography{biblio}
\end{document}